\documentclass{article}
\usepackage{amsmath, amsthm}
\usepackage{amssymb}
\newtheoremstyle{theorem}
  {10pt}          
  {10pt}  
  {\sl}  
  {\parindent}     
  {\bf}  
  {. }    
  { }    
  {}     
\newtheorem{thm}{Theorem}[section]
\newtheorem{cor}[thm]{Corollary}
\newtheorem{lem}[thm]{Lemma}
\newtheorem{prop}[thm]{Proposition}
\newtheorem{defn}{Definition}[section]
\newtheorem{exam}[defn]{Example}
\newtheorem{rem}{Remark}[section]
\begin{document}
\title{\large\bf
On conformable fractional calulus}
\author{\small \bf T. Abdeljawad $^{a,b}$   \\ {\footnotesize $^a$ Department of Mathematics and General Sciences,
 Prince Sultan University-Riyadh-KSA}\\
 {\footnotesize $^b$ Department of Mathematics, \c{C}ankaya
University, 06530, Ankara, Turkey}}
\date{}
\maketitle {\footnotesize {\noindent\bf Abstract.} Recently, the
authors  Khalil, R., Al Horani, M., Yousef. A. and Sababheh, M., in  " A new Deﬁnition Of Fractional
Derivative, J. Comput. Appl. Math. 264. pp. 65–70, 2014.
" introduced a new simple
well-behaved definition of the fractional derivative called
conformable fractional derivative. In this article we proceed on to
develop the definitions there and set the basic concepts in this new
simple interesting  fractional calculus. The fractional versions of
chain rule, exponential functions, Gronwall's inequality,
integration by parts, Taylor power series expansions, Laplace
transforms and linear differential systems are proposed and
discussed.


\section{Introduction }
The fractional calculus \cite{book1,book2,book3} attracted many
researches in the last and present centuries. The impact of this
fractional calculus in both pure and applied branches of science and
engineering started to increase substantially during the last two
decades apparently. Many researches started to deal with the
discrete versions of this fractional calculus benefitting from the
theory of time scales ( see \cite{Gray, Miller, Ferd2, THFer, Thsh}
and the references therein. The main idea behind setting this
fractional calculus is summarized into two approaches. The first
approach is the Riemann-Liouville which based on iterating the
integral operator $n$ times and then replaced it by one integral via
the famous Cauchy formula where then $n!$ is changed to the Gamma
function and hence the fractional integral of noninteger order is
defined. Then integrals were used to define Riemann and Caputo fractional
derivatives. The second approach is the Gr\"{u}nwald-Letnikov
approach which based on iterating the derivative $n$ times and then
fractionalizing by using the Gamma function in the binomial
coefficients. The obtained fractional derivatives in this calculus
seemed complicated and lost some of the basic properties that usual
derivatives have such as the product rule and chain rule. However,
the semigroup properties of these fractional operators behave well
in some cases. Recently, the author in \cite{New} define a new
well-behaved simple fractional derivative called "the conformable
fractional derivative" depending just on the basic limit definition
of the derivative. Namely. for a function $f:(0,\infty)\rightarrow
\mathbb{R} $ the (conformable) fractional derivative of order $0<
\alpha \leq 1$ of $f$ at $t>0$ was defined by

\begin{equation}\label{basic}
  T_\alpha f(t)= \lim_{\epsilon\rightarrow 0}\frac{f(t+ \epsilon t^{1-\alpha } )-f(t)}  {\epsilon
  },
\end{equation}
and the fractional derivative at $0$ is defined as $(T_\alpha
f)(0)=\lim_{t \rightarrow 0^+}(T_\alpha  f)(t)$.

They then define the fractional derivative of higher order (i.e. of
order $\alpha > 1$) as we will see below in next sections. They also
define the fractional integral of order $0< \alpha \leq 1$ only.
They then proved the product rule, the fractional mean value theorem
solved some (conformable) fractional differential equations where
the fractional exponential function $e^{\frac{t^\alpha}{\alpha}}$
played an important rule. While in case of well-known fractional
calculus Mittag-Leffler functions generalized exponential functions.
In this article we continue to settle the basic definitions and
concepts of this new theory motivated by the fact that there are
certain functions which do not have  Taylor power series
representation or their Laplace transform can not be calculated and
so forth but will be possible to do so by the help of the theory of
this conformable fractional calculus. The article is organized as
follows: In Section 2 the left and right (conformable) fractional
derivatives and fractional integrals of higher orders are defined,
the fractional chain rule and Gronwall inequality are obtained and
the action of fractional derivatives and integrals to each other are
discussed. The conformable and sequential conformable fractional derivatives of higher orders are discussed at the end points as well . In Section 3, two kinds of fractional integration by
parts formulas when $0<\alpha \leq$ are obtained where the usual
integration by parts formulas in usual cases are reobtained when
$\alpha \rightarrow 1$. In Section 4, the fractional power series
expansions for certain functions that do not have Taylor power
series representation in in usual calculus are obtained and the
fractional Taylor inequality is proved. Finally, in Section 5 the
fractional Laplace transform is defined and used to solve a
conformable fractional linear differential equation, where also the
fractional Laplace of certain basic functions are calculated.

\section{Basic definitions and tools}

\begin{defn} \label{de}
The (left) fractional derivative starting from $a$ of a function $f:[a,\infty)$ of order $0< \alpha\leq 1$ is defined by

\begin{equation}\label{deq}
    (T_\alpha ^a f)(t)= \lim_{\epsilon \rightarrow 0} \frac{f(t+ \epsilon(t-a)^{1-\alpha } )-f(t)}  {\epsilon }.
\end{equation}
When $a=0$ we write $T_\alpha$. If $(T_\alpha f)(t)$ exists on
$(a,b)$ then $(T_\alpha ^ a f)(a)=\lim_{t \rightarrow a^+}(T_\alpha
^ a f)(t)$ .

\indent

The (right) fractional derivative of order $0< \alpha\leq 1$ terminating at $b$ of $f$ is defined by

\begin{equation}\label{right}
    (~^{b}_{\alpha}T f)(t)=-\lim_{\epsilon \rightarrow 0} \frac{f(t+ \epsilon(b-t)^{1-\alpha } )-f(t)}  {\epsilon }.
\end{equation}
 If $(~^{b}T_\alpha f)(t)$ exists on
$(a,b)$ then $(~^{b}T_\alpha f)(b)=\lim_{t \rightarrow
b^-}(~^{b}T_\alpha f)(t)$ .
\end{defn}

Note that if $f$ is differentiable then $(T_\alpha ^a f)(t)=
(t-a)^{1-\alpha} f^\prime (t)$ and $(~^{b}_{\alpha}T
f)(t)=-(b-t)^{1-\alpha}f^\prime (t)$. It is clear that the
conformable fractional derivative of the constant function is zero.
Conversely, if $T_\alpha f(t)=0$  on an interval $(a,b)$ then by the
help of conformable fractional mean value theorem proved in
\cite{New} we can easily show that $f(x)=0$ for all $x \in (a,b)$.
Also, by the help of the fractional mean value theorem there we can
show that if the conformable fractional derivative of a function $f$
on an interval $(a,b)$ is positive (negative) then it the graph of
$f$ is increasing (decreasing) there.

\textbf{Notation} $(I_\alpha^a f)(t)=\int_{a}^tf(x) d\alpha(x,a)=\int_{a}^t (x-a)^{\alpha-1}f(x)dx$. When $a=0$ we write $d\alpha(x)$. Similarly, in the right case we have
$(~^{b}I_\alpha f)(t)=\int_{t}^bf(x) d\alpha(b,x)=\int_{a}^t (b-x)^{\alpha-1}f(x)dx$. The operators $I_\alpha^a $ and $~^{b}I_\alpha$ are called conformable left and right fractional integrals of order $0<\alpha \leq 1$.

In the higher order case we can generalize to the following:
\begin{defn}
Let $\alpha \in (n,n+1], $ and set $\beta= \alpha-n$. Then, the (left) fractional derivative starting from $a$ of a function $f:[a,\infty)$ of order $\alpha$, where $f^{(n)}(t)$ exists, is defined by

\begin{equation}\label{hde}
    (\textbf{T}_\alpha ^a f)(t)= (T_\beta ^a f^{(n)})(t).
\end{equation}
When $a=0$ we write $\textbf{T}_\alpha$.

The (right) fractional derivative of order $\alpha$ terminating at $b$ of $f$ is defined by

\begin{equation}\label{right}
(~^{b}_{\alpha}\textbf{T} f)(t)=(-1)^{n+1}(~^{b}_{\beta}T
f^{(n)})(t)
\end{equation}
\end{defn}
Note that if $\alpha=n+1$ then $\beta=1$ and the fractional derivative of $f$ becomes $f^{(n+1)}(t)$. Also when $n=0$ (or $\alpha \in (0,1)$) then $\beta=\alpha$ and the definition coincides with those in Definition \ref{de}.

\indent

\begin{lem} \label{n} \cite{New}
Assume that $f:[a,\infty)\rightarrow \mathbb{R}$ is continuous and
$0< \alpha\leq 1$. Then, for all $t>a$ we have $$T_\alpha^a I_\alpha
^a f(t)=f(t).$$
\end{lem}

In the right case we can similarly prove:

\begin{lem} \label{nr}
Assume that $f:(-\infty,b]\rightarrow \mathbb{R}$ is continuous and
$0< \alpha\leq 1$. Then, for all $t<b$ we have $$~^{b}T_\alpha
~^{b}I_\alpha  f(t)=f(t).$$
\end{lem}

 Next we give the definition of left and
right fractional integrals of any order $\alpha > 0$.

\begin{defn}\label{int}
Let $\alpha \in (n,n+1]$ then the left fractional integral starting at $a$ if order $\alpha$ is defined by
\begin{equation}\label{intq}
    (I_\alpha^a f )(t)= \textbf{I}_{n+1} ^a((t-a)^{\beta-1}f)=\frac{1}{n!}\int_a^t (t-x)^{n}(x-a)^{\beta-1}f(x)dx
\end{equation}
\end{defn}

Notice that if $\alpha=n+1$ then $\beta=\alpha-n=n+1-n=1$ and hence
$(I_\alpha ^a f)(t)=(\textbf{I}_{n+1}^a f)(t)=\frac{1}{n!}\int_a^t
(t-x)^n f(x)dx$,
 which is by means of Cauchy formula the iterative integral of $f$,  $n+1$ times  over $(a,t]$.

 Recalling that the left Rieamann-Liouville fractional integral of
 or order $\alpha >0$ starting from $a$ is defined by

 \begin{equation}\label{rl}
    (\textbf{I}_\alpha ^a f)(t)=\frac{1}{\Gamma(\alpha)} \int_a^t
    (t-s)^{\alpha-1} f(s)ds,
 \end{equation}

we see that $(I_\alpha ^a f)(t)= (\textbf{I}_\alpha ^a f)(t)$ for
$\alpha= n+1,~~~~n=0,1,2,...$.

\begin{exam}
Recalling that \cite{book1} $\textbf{I}_\alpha ^a
(t-a)^{\mu-1})(x)=\frac{\Gamma(\mu)}{\Gamma(\mu+\alpha)}
(x-a)^{\alpha+\mu-1},~~~~\alpha, \mu >0$,

we can calculate the (conformable) fractional integral of
$(t-a)^\mu$ of order $\alpha \in (n,n+1]$. Indeed, if $\mu \in
\mathbb{R}$ such that $\alpha+\mu-n>0$ then

\begin{equation}\label{calculate}
(I_\alpha^a (t-a)^\mu)(x)=(\textbf{I}_{n+1}^a
(t-a)^{\mu+\alpha-n-1})(x)=\frac{\Gamma(\alpha+\mu-n)}{\Gamma(\alpha+\mu+1)}
(x-a)^{\alpha+\mu}.
\end{equation}
Analogously, we can find the (conformable) right fractional integral
of such functions. Namely,
\begin{equation}\label{calculater}
(~^{b}I_\alpha (b-t)^\mu)(x)=(~^{b}\textbf{I}_{n+1}
(t-a)^{\mu+\alpha-n-1})(x)=\frac{\Gamma(\alpha+\mu-n)}{\Gamma(\alpha+\mu+1)}
(b-x)^{\alpha+\mu},
\end{equation}
where  $\mu \in \mathbb{R}$ such that $\alpha+\mu-n>0$.
\end{exam}
From the above discussion, we notice that the Riamann fractional
integrals and conformable fractional integrals of polynomial
functions are different up to a constant multiple and coincide for
natural orders.

The following semigroup property relates the composition operator
$I_\mu I_\alpha$ and the operator $I_{\alpha +\mu}$.

\begin{prop}
Let $f:[a,\infty)\rightarrow \mathbb{R}$ be a function and
$0<\alpha,~\mu \leq 1$ be such that $1< \alpha+\mu \leq 2$. Then
\begin{equation}\label{semi}
    (I_\alpha I_\mu f)(t)=\frac{t^\mu}{\mu} (I_\alpha f)(t)+
    \frac{1}{\mu}(I_{\alpha + \mu} f)(t)-\frac{t}{\mu} \int_0^t
    s^{\alpha+\mu-2} f(s) ds.
\end{equation}
\begin{proof}
Interchanging the order of integrals and noting that
\begin{equation}\label{inter}
    (I_{\alpha+\mu}f)(t)=(\textbf{I}_2 s^{\alpha+\mu-2} f(s))(t)=\int_0^t (t-s) s^{\alpha+\mu-2}ds,
\end{equation}
we see that
\begin{eqnarray} \nonumber
  (I_\alpha I_\mu f)(t) &=& \int_0^t (\int_0^{t_1} f(s)s^{\alpha-1}ds)t_1^{\mu-1} dt_1 \\ \nonumber
  &=&\int_0^t f(s)s^{\alpha-1}(\int_s^{t} t_1^{\mu-1}dt_1) ds \\ \nonumber
   &=& \int_0^t f(s)s^{\alpha-1}[\frac{t^\mu}{\mu}-\frac{s^\mu}{\mu}] ds\\ \nonumber
   &=& \frac{t^\mu}{\mu} (I_\alpha f)(t)+ \frac{1}{\mu}[(I_{\alpha+\mu}f)(t)-t\int_0^t s^{\alpha+\mu-2}f(s) ds] \\
\end{eqnarray}

\end{proof}
Notice that if in (\ref{semi}) we let $\alpha,\mu\rightarrow 1$ we
verify that $(I_1 I_1 f)(t)=(I_2f)(t)$.
\end{prop}

Recalling  the action of the $Q-$operator on fractional integrals
($Qf(t)=f(a+b-t),~~f:[a,b]\rightarrow\mathbb{R}$) on Riemann left
and right  fractional integrals:
\begin{equation}\label{Q}
    Q\textbf{I}_\alpha ^a f(t)= ~^{b}\textbf{I}_\alpha Qf(t),
\end{equation}
we see that for $\alpha \in (n,n+1]$

\begin{equation}\label{QQ}
    Q I_\alpha ^a f(t)= Q \textbf{I}_{n+1}^a((t-a)^{\alpha-n-1}f(t))= ~^{b}\textbf{I}_{n+1} ((b-t)^{\alpha-n-1} f(a+b-t))  =~^{b}I_\alpha Qf(t),
\end{equation}

 Now we give the generalized version of Lemma \ref{n}.
\begin{lem} \label{nn}
Assume that $f:[a,\infty)\rightarrow \mathbb{R}$ such that
$f^{(n)}(t)$ is continuous and $\alpha \in (n,n+1]$. Then, for all
$t>a$ we have
$$\textbf{T}_\alpha^a I_\alpha ^a f(t)=f(t).$$
\end{lem}

\begin{proof}
From the definition we have

\begin{equation}
\textbf{T}_\alpha^a I_\alpha ^a f(t)= T_\beta ^a (\frac{d^n}{dt^n}
I_\alpha ^a f(t))=T_\beta ^a (\frac{d^n}{dt^n} I_{n+1} ^a (
(t-a)^{\beta-1}f(t)  ))=T_\beta ^a ( I_{1} ^a ( (t-a)^{\beta-1}f(t)
))
\end{equation}
That it is $\textbf{T}_\alpha^a I_\alpha ^a f(t)=T_\beta ^a I_\beta
^a f(t) $ and hence the result follows by Lemma \ref{n}.
\end{proof}
Similarly we can generalize Lemma \ref{nr}. Indeed,
\begin{lem} \label{nnr}
Assume that $f:(-\infty,b]\rightarrow \mathbb{R}$ such that
$f^{(n)}(t)$ is continuous and $\alpha \in (n,n+1]$. Then, for all
$t<b$ we have
$$~^{b}\textbf{T}_\alpha ~^{b}I_\alpha  f(t)=f(t).$$
\end{lem}

\begin{lem} \label{ll}
Let $f,h:[a,\infty)\rightarrow \mathbb{R}$ be functions such that
$T_\alpha ^a$ exists for $t>a$, $f$ is differentiable on
$(a,\infty)$ and $T_\alpha ^a f(t)=(t-a)^{1-\alpha} h(t)$. Then
$h(t)=f^\prime (t)$ for all $t>a$.
\end{lem}
The proof follows by definition and setting $h=\epsilon
(t-a)^{1-\alpha}$ so that $h\rightarrow 0$ as $\epsilon\rightarrow
0$. As result of Lemma \ref{ll} we can state

\begin{cor}
$f:[a,b)\rightarrow \mathbb{R}$ be such that $(I_\alpha ^a T_\alpha
^a)f(t)$ exists for $b>t>a$. Then, $f(t)$ is differentiable on
$(a,b)$.
\end{cor}
\begin{lem} \label{op}
Let $f:(a,b) \rightarrow \mathbb{R}$ be differentiable and $0<
\alpha \leq 1$. Then, for all $t>a$ we have
\begin {equation}
I_\alpha^a
 T_\alpha ^a(f)(t)= f(t)-f(a).
  \end{equation}
\end{lem}
\begin{proof}
Since $f$ is differentiable then by the help of  Theorem 2.1 (6) in
\cite{New} we have

\begin{equation}
I_\alpha^a
 T_\alpha ^a (f)(t)= \int_a^t
(x-a)^{\alpha-1} T_\alpha (f)(x) dx= \int_a^t (x-a)^{\alpha-1}(x-a)^{1-\alpha}
f^\prime(x) dx=f(t)-f(a).
\end{equation}
\end{proof}
The above Lemma \ref{op} above can be generalized for the higher
order as follows.

\begin{prop}
Let $\alpha \in (n,n+1] $ and $f:[a,\infty) \rightarrow \mathbb{R}$
be $(n+1)$ times differentiable for $t>a$. Then, for all $t>a$ we
have
\begin {equation} \label{c1}
I_\alpha^a
 \textbf{T}_\alpha ^a(f)(t)= f(t)-\sum_{k=0}^n \frac{f^{(k)}(a)(t-a)^k} {k!}.
  \end{equation}
\end{prop}

\begin{proof}
From definition and  Theorem 2.1 (6) in
\cite{New}  we have

\begin{equation}\label{b}
 I_\alpha^a\textbf{T}_\alpha ^a(f)(t)= I_{n+1}^a ((t-a)^{\beta-1} T_\beta ^af^{(n)}(t))=   I_{n+1}^a ((t-a)^{\beta-1}(t-a)^{1-\beta} f^{(n+1)}(t))=I_{n+1}^a  f^{(n+1)}(t).
\end{equation}
Then integration by parts gives (\ref{c1}).
\end{proof}

Analogously, in the right case we have

\begin{prop} \label{opr}
Let $\alpha \in (n,n+1] $ and $f:(-\infty,b] \rightarrow \mathbb{R}$
be $(n+1)$ times differentiable for $t<b$. Then, for all $t<b$ we
have
\begin {equation} \label{c1}
~^{b}I_\alpha
 ~^{b}\textbf{T}_\alpha (f)(t)= f(t)-\sum_{k=0}^n \frac{(-1)^k f^{(k)}(b)(b-t)^k} {k!}.
  \end{equation}
  In particular, if $n=0$ or $0< \alpha \leq 1$, then $~^{b}I_\alpha ~^{b}T_\alpha
  (f)(t)=f(t)-f(b)$.
\end{prop}

\begin{thm} \label{Ch} (\textbf{Chain Rule})
Assume $f,g:(a,\infty)\rightarrow \mathbb{R}$ be (left)
$\alpha-$differentiable functions, where $0< \alpha \leq 1$. Let
$h(t)=f(g(t))$. Then $h(t)$ is (left)$\alpha-$differentiable and for
all $ t ~\texttt{with}~t \neq a ~\texttt{and}~g(t)\neq 0$ we have

\begin{equation}
(T_\alpha ^a h)(t)= (T_\alpha ^a f)(g(t)).( T_\alpha ^a g)(t).
g(t)^{\alpha-1}.
\end{equation}
If $t=a$ we have
\begin{equation}
(T_\alpha ^a h)(a)= \lim_{t\rightarrow a ^+}(T_\alpha ^a f)(g(t)).(
T_\alpha ^a g)(t). g(t)^{\alpha-1}.
\end{equation}
\end{thm}
\begin{proof}
By setting $u=t+ \epsilon(t-a)^{1-\alpha}$ in the definition and
using continuity of $g$ we see that

\begin{eqnarray}
  T_\alpha^a h(t)=  &=& \lim_{u\rightarrow t}
  \frac{f(g(u))-f(g(t))}{(u-t)}t^{1-\alpha}\\ \nonumber
   &=&\lim_{u\rightarrow t} \frac{f(g(u))-f(g(t))}{(g(u)-g(t))}. \lim_{u\rightarrow t}
   \frac{g(u)-g(t)}{u-t}t^{1-\alpha}\\ \nonumber
   &=&\lim_{g(u)\rightarrow g(t)} \frac{f(g(u))-f(g(t))}{(g(u)-g(t))}.g(t)^{1-\alpha} .T_\alpha ^a
   g(t).g(t)^{\alpha-1}\\ \nonumber
   &=&(T_\alpha ^a f)(g(t)).( T_\alpha ^a g)(t).
g(t)^{\alpha-1}.
\end{eqnarray}
\end{proof}

\begin{prop}
Let $f:[a,\infty)\rightarrow \infty$ be twice differentiable on
$(a,\infty)$ and $0<\alpha,~\beta \leq 1$ such that
$1<\alpha+\beta\leq 2$. Then
\begin{equation}\label{ff}
    (T_\alpha ^a T_\beta ^a f)(t)= T_{\alpha+\beta}^a f(t)+
    (1-\beta) (t-a)^{-\beta}T_\alpha^a f(t).
\end{equation}
\begin{proof}
By the fractional product rule and that $f$ is twice differentiable
differentiable we have
\begin{eqnarray}
  (T_\alpha ^a T_\beta ^a f)(t) &=&
  t^{1-\alpha}\frac{d}{dt}[t^{1-\beta}(t-a)^{-\beta}f^\prime(t)]\\ \nonumber
   &=& t^{1-\alpha}[t^{1-\beta} f^{\prime\prime}(t)+(1-\beta)(t-a)^{-\beta}f^\prime(t)]
   \\ \nonumber
  &=& T_{\alpha+\beta}^a f(t)+
    (1-\beta)(t-a)^{-\beta} T_\alpha^a f(t).
\end{eqnarray}
\end{proof}
\end{prop}
Note that in (\ref{ff}) if we let $\alpha,\beta\rightarrow 1$ then
we have $T_\alpha ^a T_\beta f(t)= T_2 f(t)=f^{\prime\prime}(t).$

 Next we prove a fractional
version of \textbf{Gronwall inequality} which will be useful is
studying stability of (conformable) fractional systems.

\begin{thm}
Let $r$ be a continuous, nonnegative function on an interval
$J=[a,b]$ and $\delta$ and k be nonnegative constants such that
\begin{equation} \nonumber
r(t)\leq \delta + \int_a^t k r(s)(s-a)^{\alpha-1}ds~~~~~~~~(t \in
J).
\end{equation}
Then for all $t \in J$
\begin{equation} \nonumber
r(t)\leq \delta e^{k \frac{(t-a)^\alpha}{\alpha}}.
\end{equation}

\end{thm}

\begin{proof}
Define $R(t)= \delta + \int_a^t k r(s)(s-a)^{\alpha-1}ds= \delta +
I_\alpha^a ( k r(s))(t)$. Then $R(a)=\delta$ and $R(t)\geq r(t)$,
and
\begin{equation} \label{G1}
T_\alpha ^a R(t)-k R(t)=k r(t)-k R(t)\leq k r(t)-kr(t)=0.
\end{equation}
Multiply (\ref{G1}) by $K(t)=e^{-k \frac{(t-a)^\alpha}{\alpha}}$. By
the help of chain rule in Theorem \ref{Ch} we see that $T_\alpha ^a
K(t)= -k K(t)$ and hence by the product rule we conclude that
$T_\alpha^a(K(t)R(t))\leq 0$. Since $K(t)R(t)$ is differentiable on
$(a,b)$ then Lemma \ref{op} implies that

\begin{equation}
I_\alpha ^a T_\alpha^a (K(t)R(t))=K(t)R(t)-K(a)R(a)=K(t)R(t)-\delta
\leq 0.
\end{equation}
Hence

\begin{equation}
r(t)\leq R(t) \leq \frac{\delta}{K(t)}=\delta e^{k
\frac{(t-a)^\alpha}{\alpha}}.
\end{equation}
 Which completes the proof.
\end{proof}
Finally, in this section we discuss the conformable fractional derivative at $a$ in the left case and at $b$ at the right case for some smooth functions. Let $n-1<\alpha <n$ and assume $f:[a,\infty)\rightarrow \mathbb{R}$ be such that $f^{(n)}(t)$  exists and continuous. Then, $(\textbf{T}_\alpha^a f)(t)=(T_{\alpha+1-n}^a f^{(n-1)})(t)=(t-a)^{n-\alpha}f^{(n)}(t)$ and thus $(\textbf{T}_\alpha^a f)(a)=\lim_{t\rightarrow a^+} (t-a)^{n-\alpha}f^{(n)}(t)=0$. Similarly, in the right case we have $(~^{b}\textbf{T}_\alpha f)(b)=\lim_{t\rightarrow b^-} (b-t)^{n-\alpha}f^{(n)}(t)=0$, for $(-\infty,b]\rightarrow \mathbb{R}$ with $f^{(n)}(t)$  exists and continuous.
Now, let $0<\alpha <1$  and $n \in \{1,2,3,...\}$ then the left (right) sequential conformable fractional derivative of order $n$ is defined by
\begin{equation}\label{seq}
 ~^{(n)} T_ \alpha ^a f(t)=\underbrace{T_\alpha^a T_\alpha^a...T_\alpha^a}_{n-times} f(t)
\end{equation}

and
\begin{equation}\label{rseq}
 ~^{b} T^{(n)}_ \alpha  f(t)=\underbrace{~^{b}T_\alpha ~^{b}T_\alpha...~^{b}T_\alpha}_{n-times} f(t),
\end{equation}
respectively.
If $f:[a,\infty)\rightarrow \mathbb{R}$ is second continuously differentiable and $0<\alpha \leq \frac{1}{2}$ then direct calculations shows that
\begin{equation}\nonumber
~^{(2)}T_\alpha ^a (t) = T_\alpha ^a T_\alpha ^a f(t)=\left\{ \begin{array}{ll} (1-\alpha)(t-a)^{1-2\alpha}f^\prime (t)+(t-a)^{2-2\alpha} f^{\prime\prime}(t) & if\; t>a,\\
0 & if\;~ t=a.\end{array}
    \right.
\end{equation}
Similarly, in the right case, for $f:(-\infty,b]\rightarrow \mathbb{R}$ is second continuously differentiable and $0<\alpha \leq \frac{1}{2}$ then direct calculations show that
\begin{equation}\nonumber
~^{b}T^{(2)}_\alpha  (t) =~^{b} T_\alpha ~^{b} T_\alpha  f(t)=\left\{ \begin{array}{ll} (1-\alpha)(b-t)^{1-2\alpha}f^\prime (t)+(b-t)^{2-2\alpha} f^{\prime\prime}(t) & if\; t<b,\\
0 & if\;~ t=b.\end{array}
    \right.
\end{equation}
This shows that the second order sequential conformable fractional derivative may not be continuous even $f$ is second continuously differentiable for $\frac{1}{2}<\alpha <1$.
If we proceed inductively, then we can see that if $f$ is $n-$continuously continuously differentiable and $0< \alpha \leq \frac{1}{n}$ then the $n-$th order sequential conformable fractional derivative is continuous and vanishes at the end points ($a$ in the left case and $b$ in the right case).
\section{Integration by parts}

\begin{thm} \label{by parts}
Let $f,g:[a,b]\rightarrow \mathbb{R}$ be two functions such that
$fg$ is differentiable. Then
\begin{equation}
\int_a^b f(x) T_\alpha ^a(g)(x) d\alpha(x,a)= fg|_a^b- \int_a^b g(x)
T_\alpha ^a(f)(x) d\alpha(x,a)
\end{equation}
\end{thm}
The proof followed by Lemma \ref{op} applied to $fg$ and Theorem 2.1
(3) in \cite{New}.

The following integration by parts formula is by means of left and
right fractional integrals.

\begin{prop} \label{bypi}
Let $f,g:[a,b]\rightarrow \mathbb{R}$ be functions and $0< \alpha
\leq 1$. Then
\begin{equation}
\int_a^b (I_\alpha^a f)(t) g(t) d_\alpha (b,t)=\int_a^b f(t) (~^{b}
I_\alpha g)(t) d_\alpha (t,a).
\end{equation}
\end{prop}
\begin{proof}
From definition we get
\begin{equation}
\int_a^b (I_\alpha^a f)(t) g(t) d_\alpha (t,a)=\int_a^b (\int_a^t
(x-a)^{\alpha-1} f(x)dx) g(t) (b-t)^{\alpha-1} dt.
\end{equation}
Interchanging the order of integrals we reach at

$$\int_a^b (I_\alpha^a f)(t) g(t) d_\alpha (b,t)=\int_a^b f(x) (~^{b}
I_\alpha g)(x) d_\alpha (x,a). $$ Which completes the proof.
\end{proof}

Next we employ Proposition \ref{bypi} to prove an integration by
parts formula by means of left and right fractional derivatives.

\begin{thm} \label{by parts2}
Let $f,g:[a,b]\rightarrow \mathbb{R}$ be differentiable functions
and $0< \alpha \leq 1$. Then
\begin{equation}
\int_a^b (T_\alpha^a f)(t) g(t) d_\alpha (t,a)=\int_a^b f(t) (~^{b}
T_\alpha g)(t) d_\alpha (b,t)+f(t)g(t)|_a^b.
\end{equation}
\begin{proof}
By Proposition \ref{opr} and that $g$ is differentiable, we have

\begin{equation}
\int_a^b (T_\alpha^a f)(t) g(t) d_\alpha (t,a)=\int_a^b (T_\alpha^a
f)(t)~ ^{b}I_\alpha ~ ^{b}T_\alpha g(t) d_\alpha (t,a)+g(b) \int_a^b
(T_\alpha^a f)(t)d_\alpha (t,a).
\end{equation}
Applying Proposition \ref{bypi} leads to

\begin{equation}
\int_a^b (T_\alpha^a f)(t) g(t) d_\alpha (t,a)=\int_a^b (I_\alpha ^a
T_\alpha^a f)(t) ~ ^{b}T_\alpha g(t) d_\alpha (b,t)+g(b) (I_\alpha
^a T_\alpha^a f)(a).
\end{equation}
Then the proof is completed by the help of Lemma \ref{op} by
substituting  $(I_\alpha ^a T_\alpha^a f)(t)= f(t)-f(a)$ using that
$f$ is differentiable and by the help of Proposition \ref{opr} and
that $g$ is differentiable by substituting $(~^{b}I_\alpha
~^{b}T_\alpha g)(t)= g(t)-g(b)$.
\end{proof}
\end{thm}
\begin{rem}
Notice that if in Theorem \ref{by parts} or Theorem \ref{by parts2}
 we let $\alpha\rightarrow 1$ then we obtain the integration by
parts formula in usual calculus, where we have to note that
$d_\alpha(t,a)\rightarrow dt~$, $d_\alpha(b,t)\rightarrow dt,~$
$T_\alpha ^a f(t) \rightarrow f^\prime(t)$ and $~^{b}T_\alpha  f(t)
\rightarrow -f^\prime(t)$ as $\alpha \rightarrow 1$.

\end{rem}

In Theorem \ref{by parts} and  Theorem \ref{by parts2} we needed
some differentiability conditions. We next define some function
spaces on which the obtained integration by parts formulas are still
valid.

\begin{defn}
For $0< \alpha\leq 1$ and an interval $[a,b]$ define

$$I_\alpha ([a,b])=\{f:[a,b]\rightarrow \mathbb{R}: f(x)= (I_\alpha ^a \psi)(x)+f(a),~\texttt{for some}~\psi \in L_\alpha (a)\},$$

and

$$~^{\alpha}I ([a,b])=\{g:[a,b]\rightarrow \mathbb{R}: g(x)= (~^{b}I_\alpha \varphi)(x)+g(b),~\texttt{for some}~\psi \in L_\alpha (b)\},$$

where

$$L_\alpha(a)=\{\psi:[a,b]\rightarrow \mathbb{R}\}: (I_\alpha ^a \psi)(x)~\texttt{exists for all}~x \in [a,b]\}, $$

and

$$L_\alpha(b)=\{\varphi:[a,b]\rightarrow \mathbb{R}\}: (~^{b}I_\alpha  \varphi)(x)~\texttt{exists for all}~x \in [a,b]\}. $$

\end{defn}

\begin{lem} \label{m}
Let
 $f,g:[a,b]\rightarrow \mathbb{R}$ be  functions and $0<\alpha \leq 1$. Then

 a) If $f$ is left (g is right) $\alpha-$ differentiable then $f \in I_\alpha ([a,b])$ ($ g \in ~_{\alpha}I ([a,b])$).

 \indent

 b) If $f \in I_\alpha ([a,b])$ with $f(x)=(I_\alpha^a \psi)(x)+f(a)$
 where $\psi$ is continuous  then $\psi(x)=T_\alpha ^a f(x)$ and $(I_\alpha ^a  T_\alpha ^a
 f)(x)=f(x)-f(a)$.

 \indent

 c)If $g \in ~_{\alpha}I ([a,b])$ with $g(x)=(~^{b}I_\alpha \varphi)(x)+g(b)$
 where $\varphi$ is continuous  then $\varphi(x)=~^{b}T_\alpha g(x)$ and  $(~^{b}I_\alpha   ~^{b}T_\alpha
 g)(x)=g(x)-g(b)$.
\end{lem}

\begin{proof}
The proof of a) follows by Lemma \ref{op} and Proposition \ref{opr} by choosing $\psi(t)=T_\alpha^a f$ and choosing $\varphi(t)=~^{b}T_\alpha g$. The proof of b) follows by Lemma \ref{n} and the fact that the left $\alpha-$derivative of constant function is zero. The proof of c) follows by Lemma \ref{nr} and the fact that the right $\alpha-$derivative of  constant function is zero.
\end{proof}

\begin{thm} \label{by parts2g}
Let $f,g:[a,b]\rightarrow \mathbb{R}$ be  functions such that $f \in I_\alpha([a,b])$ with $\psi(t)$ is continuous and $g \in ~_{\alpha}I ([a,b])$ with $\varphi(t)$ is continuous
and $0< \alpha \leq 1$. Then
\begin{equation}
\int_a^b (T_\alpha^a f)(t) g(t) d_\alpha (t,a)=\int_a^b f(t) (~^{b}
T_\alpha g)(t) d_\alpha (b,t)+f(t)g(t)|_a^b.
\end{equation}
\end{thm}
\begin{proof}
The proof is similar to that in Theorem \ref{by parts2} where we make use of b) and c) in Lemma \ref{m}.
\end{proof}
\section{Fractional power series expansions}
Certain functions, being not infinitely differentiable at  some
point, do not have Taylor power series expansion there. In this
section we set the fractional power series expansions so that those
functions will have fractional power series expansions.

\begin{thm}
Assume $f$ is an infinitely $\alpha-$differentiable function, for
some $0< \alpha \leq 1$ at a neighborhood of a point $t_0$. Then $f$
has the fractional power series expansion:

\begin{equation}\label{exp}
f(t)= \sum_{k=0}^\infty \frac {(T_\alpha
^{t_0}f)^{(k)}(t_0)(t-t_0)^{k \alpha}}{\alpha^k k!},~~~~t_0<t<t_0+
R^{1/\alpha},~~~R>0.
\end{equation}
 Here $(T_\alpha ^{t_0}f)^{(k)}(t_0)$ means the
application of the fractional derivative $k$ times.
\end{thm}

\begin{proof}
Assume $f(t)=c_0+ c_1 (t-t_0)^\alpha+ c_2 (t-t_0)^{2 \alpha }+ c_3
(t-t_0)^{3 \alpha }+...,~~~t_0<t<t_0+ R^{1/\alpha},~~~R >0.$
\end{proof}
Then, $f(t_0)=c_0$. Apply $T_\alpha ^{t_0}$ to $f$ and evaluate at
$t_0$ we see that $(T_\alpha ^{t_0}f)(t_0)=c_1 \alpha$ and hence
$c_1=\frac {(T_\alpha ^{t_0}f)(t_0)} {\alpha}$. Proceeding
inductively and applying $T_\alpha ^{t_0}$ to $f$  $n-$times and
evaluating at $t_0$ we see that $(T_\alpha ^{t_0} f)^{(n)}(t_0)=
c_n. \alpha (2 \alpha)...(n\alpha)=\alpha^n .n!$ and hence

\begin{equation}
c_n= \frac{(T_\alpha ^{t_0}f)^{(n)}(t_0)}   {\alpha^n .n!}.
\end{equation}
Hence (\ref{exp}) is obtained and the proof is completed.

\begin{prop}(\textbf{fractional Taylor inequality})
Assume $f$ is an infinitely $\alpha-$differentiable function, for
some $0< \alpha \leq 1$ at a neighborhood of a point $t_0$ has the Taylor power series representation (\ref{exp}) such that $|(T_\alpha^a f)^{n+1}|\leq M,~~~M>0$ for some $n \in \mathbb{N}$. Then, for all $(t_0,t_0+R)$

\begin{equation}\label{I1}
    |R_n^\alpha(t)|\leq \frac{M}{\alpha^{n+1}(n+1)!}(t-t_0)^{\alpha(n+1)},
\end{equation}
where $R_n^\alpha(t)= \sum_{k=n+1}^\infty \frac {(T_\alpha
^{t_0}f)^{(k)}(t_0)(t-t_0)^{k \alpha}}{\alpha^k k!}=
f(x)-\sum_{k=0}^n \frac {(T_\alpha ^{t_0}f)^{(k)}(t_0)(t-t_0)^{k
\alpha}}{\alpha^k k!}$.
\end{prop}
The proof is similar to that in usual calculus by applying $I_\alpha^{t_0}$ instead of integration.
\begin{exam}
Consider the fractional exponential function
$f(t)=e^{\frac{(t-t_0)^\alpha}{\alpha}}$, where $0<\alpha < 1$. The
function $f(t)$ is clearly not differentiable at $t_0$ and thus it
does not have Taylor power series representation about $t_0$.
However,  $(T_\alpha ^{t_0}f)^{(n)}(t_0)=1$ for all $n$ and hence

\begin{equation}
f(t)=\sum_{k=0}^\infty \frac {(t-t_0)^{k \alpha}}{\alpha^k k!}.
\end{equation}
The ration test shows that this series is convergent to $f$ on the
interval $[t_0,\infty)$.
\end{exam}

\begin{exam}
The functions $g(t)=sin \frac{(t-t_0)^\alpha}{\alpha}$ and $h(t)=sin \frac{(t-t_0)^\alpha}{\alpha}$ do not have Taylor power
 series expansions about $t=t_0$ for $0<\alpha <1$ since they are not differentiable there. However, by the help of Theorem
 \ref{exp} and that $T_\alpha ^{t_0}sin \frac{(t-t_0)^\alpha}{\alpha}=cos
\frac{(t-t_0)^\alpha}{\alpha}$ and
   $T_\alpha ^{t_0}cos \frac{(t-t_0)^\alpha}{\alpha}=-sin \frac{(t-t_0)^\alpha}{\alpha}$ we can see that

\begin{equation}
sin \frac{(t-t_0)^\alpha}{\alpha}=\sum_{k=0}^\infty (-1)^{k}\frac {(t-t_0)^{(2k+1) \alpha}}{\alpha^{(2k+1)} (2k+1)!},~~~t\in [t_0,\infty).
\end{equation}
and
\begin{equation}
cos \frac{(t-t_0)^\alpha}{\alpha}=\sum_{k=0}^\infty (-1)^{k}\frac {(t-t_0)^{(2k) \alpha}}{\alpha^{(2k)} (2k)!},~~~t\in [t_0,\infty).
\end{equation}

\end{exam}

\begin{exam}
The function $f(x)=\frac{1}{1- \frac{t^\alpha}{\alpha}}$ does not have Taylor power series representation about $t=0$ for $0<\alpha <1$, since it is not differentiable there. However, by the help of Theorem \ref{exp} we can see that

\begin{equation}
\frac{1}{1- \frac{t^\alpha}{\alpha}}=\sum_{k=0}^\infty t^{\alpha k},~~~t\in [0,1).
\end{equation}
or more generally,
\begin{equation}
\frac{1}{1- \frac{(t-t_0)^\alpha}{\alpha}}=\sum_{k=0}^\infty (t-t_0)^{\alpha k},~~~t\in [t_0,t_0+1).
\end{equation}

\end{exam}

\begin{rem}
In case the function $f$ is defined on $(-\infty,a)$ and not
differentiable at $a$ then we search for its (conformal) right
fractional order derivatives $~^{a}T_\alpha$ at $a$ for some
$0<\alpha\leq 1$ and use it for our fractional Taylor series on some
$(a-R,a),~~~~ R>0$. For example the functions
$\frac{(a-t)^\alpha}{\alpha}$, $sin \frac{(a-t)^\alpha}{\alpha}$ and
so on.
\end{rem}
\section{The fractional Laplace transform}
In this section we will define the fractional Laplace transform and
use it to solve some linear fractional equations to produce the
fractional exponential function. Then, we use the method of
successive approximation to verify the solution by making use of the
fractional power series representation discussed in the above
section. Also we shall calculate the Laplace transform for certain
(fractional) type functions.
\begin{defn} \label{flap}
Let $t_0 \in\mathbb{ R}~$, $0<\alpha \leq 1$ and
$f:[t_0,\infty)\rightarrow$ be real valued function. Then the
fractional Laplace transform of order $\alpha$ starting from $a$ of
$f$ is defined by

\begin{equation}
L_\alpha^ {t_0} \{f(t)\}(s)= F_\alpha^{t_0}(s)=\int_{t_0}^\infty
e^{-s \frac{(t-t_0)^\alpha}{\alpha}}f(t)
d\alpha(t,t_0)=\int_{t_0}^\infty e^{-s \frac{(t-t_0)^\alpha }
{\alpha}}f(t)(t-t_0)^{ \alpha-1}dt
\end{equation}
\end{defn}

\begin{thm}
Let $a \in\mathbb{ R}~$, $0<\alpha \leq 1$ and
$f:(a,\infty)\rightarrow$ be differentiable real valued function.
Then
\begin{equation}\label{tran}
L_\alpha ^ a \{ T_\alpha (f)(t)\}(s)= s F_\alpha(s)-f(a).
\end{equation}
\end{thm}

\begin{proof}
The proof followed by definition, Theorem 2.1 (6) in \cite{New} and
the usual  integration by parts.
\end{proof}

\begin{exam}
Consider the conformable fractional initial value problem:
\begin{equation}
(T_\alpha ^a y)(t)= \lambda y(t),~~~~~ y(a)=y_0, ~~t>a,
\end{equation}
where the solution is assumed to be differentiable on $(a,\infty)$.

Apply the operator $I_\alpha^a$ to the above equation to obtain
\begin{equation}\label{sa1}
y(t)=y_0+ \lambda (I_\alpha^a y)(t).
\end{equation}
Then

\begin{equation}\label{sa2}
y_{n+1}=y_0 + \lambda (I_\alpha^a y_n)(t),~~n=0,1,2,...
\end{equation}
For $n=0$ we see that

\begin{equation}\label{sa3}
    y_1=y_0+ \lambda y_0 \frac{(t-a)^\alpha}{\alpha}=y_0 (1+\lambda \frac{(t-a)^\alpha}{\alpha})
\end{equation}
For $n=1$ we see that

\begin{equation}\label{sa4}
    y_2=y_0[1+ \lambda  \frac{(t-a)^\alpha}{\alpha}+\lambda ^2
    \frac{(t-a)^{2\alpha}}{\alpha(2\alpha)}].
\end{equation}
If we proceed inductively we conclude that
\begin{equation}\label{sa5}
    y_n=y_0 \sum_{k=0}^n \frac{\lambda ^k (t-a)^{k \alpha}}{\alpha^k k!}.
\end{equation}
Letting $n\rightarrow \infty$ we see that

\begin{equation}\label{sa6}
    y(t)=y_0 \sum_{k=0}^\infty \frac{\lambda ^k (t-a)^{k \alpha}}{\alpha^k k!}.
\end{equation}

Which is clearly the fractional Taylor power series representation
of the (fractional) exponential function $y_0 e^{\lambda
\frac{(t-a)^\alpha}{\alpha}}$.
\end{exam}

The following Lemma relates the fractional Laplace transform to the
usual Laplace transform:
\begin{lem}
Let $f:[t_0,\infty)\rightarrow \mathbb{R}$ be a function such that
$L_\alpha ^{t_0} \{f(t)\}(s)=F_\alpha ^{t_0}(s)$ exists. Then
\begin{equation}\label{ss}
F_\alpha ^{t_0}(s)=\mathfrak{L}\{f(t_0+ (\alpha t)^{1/\alpha})\}(s),
\end{equation}
where $\mathfrak{L}\{g(t)\}(s)=\int_0^\infty e^{-st}g(t)dt.$
\end{lem}
The proof follows easily by setting
$u=\frac{(t-t_0)^\alpha}{\alpha}$.

\begin{exam}
In this example we calculate the fractional Laplace for certain
functions.
\begin{itemize}
  \item $ L_\alpha^{t_0} \{1 \}(s)=\frac{1}{s},~~~s>0$
  \item $L_\alpha ^{t_0}\{t\}(s)=\mathfrak{L} \{ t_0+(\alpha t)^{1/\alpha}\}(s)=\frac{t_0}{s}+
   \alpha^{1/ \alpha}\frac{\Gamma(1+ \frac{1}{\alpha})}{s^{1+
   \frac{1}{\alpha}}},~~~s>0$.
   \item  $L_\alpha
   ^{0}\{t^p\}(s)=\frac{\alpha^{p/\alpha}}{s^{1+p/\alpha}}\Gamma(1+
   \frac{1}{\alpha}),~~~s>0$.
   \item $L_\alpha ^{0}\{e^{\frac{t^\alpha}{\alpha}}\}(s)= \frac{1}{s-1},~~~s>1$.
   \item $L_\alpha ^{0}\{sin\omega \frac{t^\alpha}{\alpha}\}(s)=\mathfrak{L}\{sin\omega t\}(s)=\frac{1}{\omega^2 +s^2}.$
   \item $L_\alpha ^{0}\{cos\omega \frac{t^\alpha}{\alpha}\}(s)=\mathfrak{L}\{cos\omega t\}(s)=\frac{s}{\omega^2+s^2}.$
   \item $L_\alpha ^{t_0}\{e ^{-k \frac{(t-t_0)^\alpha}{\alpha}}
   f(t)\}(s)=\mathfrak{L}\{e^{-kt}f(t_0+(\alpha
   t)^{\frac{1}{\alpha}})\}.$ For example $L_\alpha ^{0}\{e ^{-k \frac{t^\alpha}{\alpha}}
   sin \frac{t^\alpha}{\alpha}\}(s)=\mathfrak{L}\{e^{-kt}sin
   t\}(s)=\frac{1}{(s+k)^2+1}$ and $L_\alpha ^{t_0}\{e^{\lambda \frac{(t-t_0)^\alpha}{\alpha} }\}=\mathfrak{L}\{e^{\lambda
   t}\}=\frac{1}{s-\lambda}$.
\end{itemize}

\end{exam}
Notice that in the above example there are some functions, with
$0<\alpha <1$, whose usual Laplace is not easy to be calculated.
However, their fractional Laplace can be easily calculated.
\begin{exam}
We use the fractional Laplace transform to verify the solution of
the  conformable fractional initial value problem:
\begin{equation}
(T_\alpha ^a y)(t)= \lambda y(t),~~~~~ y(a)=y_0, ~~t>a,
\end{equation}
where the solution is assumed to be differentiable on $(a,\infty)$.

Apply $L_\alpha ^a$ and use (\ref{tran}) to conclude  that

\begin{equation}\label{last}
 L_\alpha ^a \{y(t)\} (s)=\frac{y_0}{s-\lambda},
\end{equation}

and hence, $y(t)=y_0 e^{\lambda \frac{(t-a)^\alpha}{\alpha}}$.
\end{exam}
Finally, we use the fractional fundamental exponential matrix to
express the solution of (conformable) fractional linear systems.

Consider the system

\begin{equation}\label{sys}
   T_\alpha ^a
   \textbf{y}(t)=A\textbf{y}(t)+\textbf{f}(t),~~~~0<\alpha \leq 1,
\end{equation}

where $\textbf{y}, \textbf{f}:[a,b)\rightarrow \mathbb{R}^n$ are
vector functions and $A$ is an $n\times n$ matrix. The general
solution of the fractional non-homogenous  system (\ref{sys}) is
express by

\begin{equation}\label{express}
    \textbf{y}(t)= e^{A \frac{(t-a)^\alpha}{\alpha}}\textbf{c}+
    \int_a^t e^{A \frac{(t-a)^\alpha}{\alpha}} e^{-A
    \frac{(s-a)^\alpha}{\alpha}} \textbf{f}(s) (s-a)^{1-\alpha} ds,
\end{equation}
where $e^{A \frac{(t-a)^\alpha}{\alpha}}=\sum_{k=0}^\infty \frac{A
^k (t-a)^{k \alpha}}{\alpha^k k!}$ and $\textbf{c}$ is a constant
vector.

\section{Some conclusions and comparisons}

\begin{enumerate}
  \item The conformable fractional derivative behaves well in the
  product rule and  chain rule while complicated formulas appear in
  case of usual fractional calculus.
  \item The conformable fractional derivative of a constant function
  is zero while it is not the case for Riemann fractional
  derivatives.
  \item Mittag-Leffler functions play important rule in fractional
  calculus as a generalization to exponential functions while the
  fractional exponential function $f(t)=e^{\frac{t^\alpha}{\alpha}}$
  appears in case of conformable fractional  calculus.
\item Conformable fractional derivatives, conformable chain rule,
conformable integration by parts,conformable Gronwall's inequality,
conformable exponential function, conformable Laplace transform and
so forth, all tend to the corresponding ones in usual calculus.
\item In case of usual calculus there some functions that do not
have Taylor power series representations about certain points but in
the theory of conformable fractional they do have.
\item \textbf{Open problem:} Is it hard to fractionalize the
conformable fractional calculus, either by iterating the confromable
fractional derivative (Gr\"{u}nwald-Letnikov approach) or by
iterating the conformable fractional integral of order $0<\alpha
\leq 1$ (Riemann approach)? Notice that when $\alpha=0$ we obtain
Hadamard type fractional integrals.
\end{enumerate}


\begin{thebibliography}{99}

\bibitem{book1}Samko G. Kilbas A. A., Marichev, Fractional Integrals
and Derivatives: Theory and Applications, Gordon and Breach, Yverdon, 1993.

\bibitem{book2}
I. Podlubny, Fractional Differential Equations, Academic Press: San
Diego CA, (1999).

\bibitem{book3} Kilbas A., Srivastava M. H.,and Trujillo J. J., Theory and Application of Fractional Differential Equations, North
Holland Mathematics Studies 204, 2006.






\bibitem{Gray} H. L. Gray and N. F.Zhang, On a new definition of the fractional difference, Mathematics of Computaion 50, (182), 513-529 (1988.)

\bibitem{Miller} K. S. Miller,  Ross B.,Fractional difference
calculus, \emph{Proceedings of the International Symposium on
Univalent Functions, Fractional Calculus and Their Applications}, Nihon University, Koriyama, Japan, (1989), 139-152.



\bibitem{Ferd2}F.M.  At{\i}c{\i}  and Eloe P. W.,  Initial value problems in
discrete fractional calculus, \emph{Proceedings of the American
Mathematical Society}, 137, (2009), 981-989.





\bibitem{THFer}T. Abdeljawad, F. At{\i}c{\i}, On the Definitions of Nabla Fractional Operators, Abstract and Applied Analysis, 2012 (2012), Article ID 406757, 13 pages, doi:10.1155/2012/406757.
\bibitem{Thsh}T. Abdeljawad, Dual identities in fractional difference calculus within Riemann, Advances in Difference Equations 2013, 2013:36 doi:10.1186/1687-1847-2013-36, arXiv:1112.5795.
\bibitem{New}Khalil, R., Al Horani, M., Yousef. A. and Sababheh, M., A new Deﬁnition Of Fractional
Derivative, J. Comput. Appl. Math. 264. pp. 65–70, 2014.

\end{thebibliography}
\end{document}